\documentclass[11pt]{article}

\usepackage[margin=1.15in]{geometry}
\usepackage{amsmath,amssymb,amsthm,mathtools}
\usepackage{microtype}
\usepackage{enumitem}
\usepackage{hyperref}
\usepackage{lmodern}
\usepackage[T1]{fontenc}

\hypersetup{colorlinks=true,linkcolor=blue,citecolor=blue,urlcolor=blue}

\makeatletter
\renewcommand{\author}[2][]{%
  \gdef\@shortauthor{#1}%
  \gdef\@author{#2}%
}
\newcommand{\subjclass}[2][2020]{%
  \gdef\@subjclassyear{#1}%
  \gdef\@subjclass{#2}%
}
\newcommand{\keywords}[1]{\gdef\@keywords{#1}}
\newcommand{\address}[1]{\gdef\@address{#1}}
\newcommand{\email}[1]{\gdef\@email{#1}}

\newcommand{\printmetadatafootnote}{%
  \begingroup
  \renewcommand{\thefootnote}{}%
  \footnotetext{
    \textit{\@subjclassyear\ Mathematics Subject Classification.} \@subjclass.\par
    \textit{Keywords and phrases.} \@keywords.
  }%
  \endgroup
}

\newcommand{\printaddress}{%
  \par\bigskip
  \begingroup\small
  \noindent\@address\par
  \noindent\textit{Email address:} \href{mailto:\@email}{\texttt{\@email}}
  \par\endgroup
}
\makeatother

\newtheorem{theorem}{Theorem}[section]
\newtheorem{proposition}[theorem]{Proposition}
\newtheorem{lemma}[theorem]{Lemma}
\newtheorem{corollary}[theorem]{Corollary}
\theoremstyle{remark}
\newtheorem{remark}[theorem]{Remark}

\newcommand{\PP}{\mathbb P}
\newcommand{\ZZ}{\mathbb Z}
\newcommand{\Pic}{\operatorname{Pic}}
\newcommand{\Aut}{\operatorname{Aut}}
\newcommand{\Nm}{\operatorname{Nm}}
\newcommand{\Ker}{\operatorname{Ker}}

\newcommand{\cO}{\mathcal O}
\newcommand{\RH}{\mathcal{RH}}
\newcommand{\cM}{\mathcal M}
\newcommand{\cA}{\mathcal A}
\newcommand{\cD}{\mathcal D}
\newcommand{\cP}{\mathcal P}
\newcommand{\wtC}{\widetilde C}

\title{Theta-duality and Prym-Torelli for cyclic covers of hyperelliptic curves}
\author[A. Shatsila]{Anatoli Shatsila}
\address{\textit{Anatoli Shatsila} \newline Doctoral School of Exact and Natural Sciences, Jagiellonian University in Krak\'ow, ul. prof. Stanisława Łojasiewicza 6, 
30-348 Kraków, Poland}
\email{anatoli.shatsila@doctoral.uj.edu.pl}
\subjclass[2020]{14H30, 14H40, 14H45, 14K12}
\keywords{Prym varieties, Prym maps, cyclic covers, hyperelliptic curves, theta-duality, trigonal pencils}
\date{}

\begin{document}
\maketitle
\printmetadatafootnote

\begin{abstract}
Let \(f:\wtC\to C\) be an \'etale cyclic cover of odd prime degree \(d\) of a hyperelliptic curve of genus $g\geq 2$. We revisit the theta-duality reconstruction argument for the generic injectivity of the associated Prym map by Naranjo-Ortega-Pirola-Spelta and identify an additional locus arising from the fixed divisor of the line bundles occurring in that argument. For \(d\ge5\), this locus does not affect reconstruction, giving injectivity under the numerical assumption \((d-1)(g-1)\ge7\). For \(d=3\), the geometry is governed instead by trigonal pencils on the quotient of $\wtC$ by a lift of the hyperelliptic involution; this yields generic degree \(2\) in genus \(5\) and injectivity for \(g\ge6\).

\end{abstract}

\section{Introduction}

The Prym-Torelli problem asks whether a covering of curves is determined by its Prym variety together with its natural polarization. We note an important difference between the generic and global versions of this problem. For the classical Prym map of \'etale double covers, generic injectivity in genus at least seven coexists with distinct coverings having the same Prym variety, as shown by the tetragonal construction (see \cite{Donagi}). Therefore it is natural to ask for Prym maps where global Prym-Torelli theorem holds.

\'Etale cyclic covers of hyperelliptic curves provide a natural setting for these questions. In odd degree, the hyperelliptic involution lifts to the covering curve and extends the cyclic action to a dihedral action. The Jacobians of the quotient curves enter the description of the Prym variety, as in the work of Ries and Ortega \cite{Ries,Ort}. This separates reconstruction into two problems: recovering a quotient curve from the polarized Prym, and recovering the distinguished map from that curve to $\PP^1$ whose Galois closure determines the original covering. We establish global reconstruction in a range of odd prime degrees and show that, in degree three, the ambiguity is precisely the choice of a simply branched trigonal map on the quotient curve.

Let $C$ be a smooth hyperelliptic curve of genus $g\ge2$, let $d=2k+1$ be an odd prime, and let $f:\wtC\to C$ be an \'etale cyclic cover of degree $d$. Its Prym variety
\[
P=P(f):=\bigl[\Ker(\Nm_f:J\wtC\to JC)\bigr]^0
\]
has dimension $(d-1)(g-1)$ and carries the polarization $\Xi_f$ obtained by restricting the canonical principal polarization of $J\wtC$. Its polarization type is
\[
\delta=(\underbrace{1,\ldots,1}_{(d-2)(g-1)},
\underbrace{d,\ldots,d}_{g-1}).
\]
We denote by $\RH_g[d]$ the moduli space of these cyclic covers and by $\cA_m^\delta$ the moduli space of polarized abelian varieties of dimension $m$ and type $\delta$. The associated Prym map is
\[
\mathcal P_g[d]:\RH_g[d]\longrightarrow\cA_{(d-1)(g-1)}^\delta.
\]

Our starting point is the theta-duality reconstruction of Naranjo--Ortega--Pirola--Spelta \cite{NOPS}. Their Theorem~1.2 states generic injectivity under the hypotheses
\[
(d-1)(g-1)\ge7,\qquad g\not\equiv3\pmod d.
\]
We revisit the theta-dual computation underlying that statement and combine it with an intrinsic recovery of the quotient Jacobian. For $d\ge5$, this gives the following improvement.

\begin{theorem}\label{thm:mainprime}
Let $d\ge5$ be an odd prime and suppose that $(d-1)(g-1)\ge7.$ Then the hyperelliptic Prym map $\mathcal P_{g}[d]$ is injective. 
\end{theorem}

The conclusion is global: no generality assumption is imposed on the covering. Thus the theorem both removes the congruence restriction and establishes global injectivity. We note that in the paper injectivity concerns isomorphism classes of smooth covers, rather than an assertion that the moduli map is an embedding.

The two improvements arise from separate steps. First, we use the description in \cite[Proposition 2.1]{BNOS} of the polarized automorphisms acting trivially on the kernel of the Prym polarization. This recovers the dihedral group acting on the Prym intrinsically. An involution then determines the principally polarized Jacobian of the quotient curve $C_0=\wtC/\langle j\rangle$, where $j$ is a lift of the hyperelliptic involution. The cyclic subgroup also recovers the endomorphisms $\beta_i=\sigma^i+\sigma^{-i}$ of $JC_0$, where $\sigma$ generates the deck group. This step does not require the covering to be general.

Second, we recover the degree-$d$ map $h:C_0\to\PP^1$ from theta-duals of the curves $\beta_i(C_0)\subset JC_0$. The argument here follows \cite{NOPS}, however, we note that the computation in \cite[Proposition 4.3]{NOPS} requires correction. In the proof of \cite[Proposition 4.3]{NOPS}, once the morphism associated with a pencil \(|L|\) is shown to factor through the degree-\(d\) map \(h:C_0\to\mathbb P^1\), it is concluded that \(d\) divides \(\deg L\). This conclusion is valid for the degree of the moving part of \(|L|\), but \(L\) may have a nontrivial fixed divisor. Keeping track of this fixed part produces an additional residual locus in the theta-dual (see Proposition \ref{prop:thetaexact} and Remark \ref{rem:fixedpart}). For \(d\geq 5\) this locus has strictly smaller dimension than the translated Brill-Noether component used in the reconstruction, so the argument can be modified to recover the cover without the congruence assumption imposed in \cite{NOPS}.

In degree \(d=3\), the situation is very different: the residual locus contains the translated Brill-Noether component. Consequently, the support of the theta-dual no longer determines the degree-three map \(C_0\to\mathbb P^1\), however, there is another geometric phenomenon which makes the reconstruction possible. The endomorphism $\beta_1=\sigma+\sigma^{-1}$ equals $-1$, and the polarized product description of the Prym depends only on the quotient curve $C_0$, which now has genus $g-1$. Conversely, the polarized Prym determines $C_0$. The trigonal map on $C_0$ is the additional datum needed to recover the cover: its $S_3$-Galois closure produces an \'etale cyclic triple cover of a hyperelliptic curve. We make this correspondence explicit and identify each Prym fiber with the set of simply branched trigonal maps on $C_0$ (see Theorem~\ref{thm:cubicfactorization}). The cubic Prym map thus has the same fibers as the forgetful map from the corresponding trigonal Hurwitz space to the moduli space of curves.

\begin{theorem}\label{thm:cubicmain}
The Prym map $\mathcal P_{g}[3]$ is injective for $g \geq 6$. For $2 \leq g \leq 5$, the general fiber of $\mathcal P_{g}[3]$ is as follows:
\begin{enumerate}[label=\textup{(\roman*)}]
\item for $g=2$ it has dimension $2$;
\item for $g=3$ it has dimension $2$;
\item for $g=4$ it has dimension $1$;
\item for $g=5$ it consists of exactly two points.
\end{enumerate}
\end{theorem}

The low-genus behavior in \textup{(i)-(iv)} is essentially due to Albano-Pirola \cite{AP}; we include the arguments to give a uniform description of the cubic fibers. Generic injectivity for $g\ge6$ also follows from \cite[Remark 4.5]{NOPS}, independently of the corrected theta-dual computation. Our contribution is the passage to global injectivity. In genus $g=5$, however, the generic degree is two, so the generic injectivity conclusion of \cite[Theorem 1.2]{NOPS} does not hold.

The exceptional degree two has a geometric interpretation. The quotient $C_0$ is a general genus-four curve, and its two trigonal pencils are exchanged by $L\mapsto K_{C_0}\otimes L^{-1}$. They arise from the two rulings of the smooth quadric containing the canonical curve and yield the two covers in a general Prym fiber. For $g\ge6$, the quotient has genus at least five and its trigonal pencil is unique, which gives generic injectivity. 

The paper is organized as follows. Section \ref{sec:intrinsic} recalls the dihedral construction and proves the intrinsic recovery of the quotient Jacobian and its endomorphisms. Section \ref{sec:theta} computes the corrected theta-dual, proves Theorem \ref{thm:mainprime}, and explains the different behavior in degree three. Section \ref{sec:d3corr} identifies the cubic Prym fibers with fibers of the trigonal Hurwitz map and proves Theorem \ref{thm:cubicmain}.

\subsection*{Acknowledgements.} 
The author was supported by the Polish National Science Centre project number 2024/53/N/ST1/01634.

\section{Preliminaries and the dihedral reconstruction}\label{sec:intrinsic}

We start by collecting some facts from \cite{Ries,Ort,NOPS}. 

Let $\iota$ be the hyperelliptic involution of $C$. It lifts to an involution $j\in\Aut(\wtC)$ satisfying
\[
j\sigma j=\sigma^{-1},
\]
where $\sigma$ is a generator of the cyclic deck group. Thus $\langle\sigma,j\rangle\simeq D_d$. Recall that all involutions in $D_d$ are pairwise conjugate. Set $C_0 := \wtC/\langle j \rangle$. For $i \in \{1,\ldots, k\}$ the map $\beta_i = \sigma^i+\sigma^{-i}$ restricts to an (unpolarized) automorphism of $JC_0$; see \cite[Proposition 4.1]{NOPS}. 

As in \cite[Section 4]{NOPS}, the Prym map $\mathcal{P}_g[d]$ factors as 
\[
\mathcal{P}_g[d]:
\RH_g[d]
\xrightarrow{\,\cP_g^1[d]\,}
\cD_k
\xrightarrow{\,\cP_g^2[d]\,}
\cA^{\delta}_{(d-1)(g-1)},
\]
where $\mathcal{D}_k$ parametrizes classes of objects $(C_0,\{\beta_1,\ldots,\beta_k\})$,  $\cP_g^1[d]$ sends $(C,\langle\eta\rangle)$ to $(C_0,\{\beta_1,\ldots,\beta_k\})$, and $\cP_g^2[d]$ sends $(C_0,\{\beta_1,\ldots,\beta_k\})$ to
$
\left(
JC_0 \times JC_0,\,
\begin{pmatrix}
2\lambda_{C_0} & \lambda_{C_0}\beta_i\\
\lambda_{C_0}\beta_i & 2\lambda_{C_0}
\end{pmatrix}
\right)
$ for any $i$,
which is isomorphic to the Prym variety $(P, \Xi)$. 

It was shown in \cite[Proposition 4.2]{NOPS} that $\cP_g^2[d]$ is generically injective. Below we offer an alternative proof of injectivity. Define
\[
\Gamma(P,\Xi):=
\{\varphi\in\Aut(P,\Xi):\varphi|_{K(\Xi)}=\operatorname{Id}\}.
\]

We shall use the following result, proved in \cite[Proposition 2.1]{BNOS}.

\begin{proposition}\label{prop:gamma}
One has
\[
\Gamma(P,\Xi)=
\{\sigma^a|_P,\,- j\sigma^a|_P: a\in\ZZ/d\ZZ\},\footnote{The group is stated as $\ZZ/d\ZZ$ in \cite[Equation (4.2)]{NOPS}; see
\cite[Remark 2.2]{BNOS}. This does not make a difference for our argument, since
$\langle\sigma\rangle$ is intrinsic in $D_d$.}
\]
and the natural map $D_d\to\Gamma(P,\Xi)$ is an isomorphism.
\end{proposition}

Since $d$ is odd, the subgroup $\langle\sigma\rangle$ is the unique subgroup of order $d$ in $D_d$, so it is intrinsic. Choose an involution $r\in\Gamma(P,\Xi)\setminus\langle\sigma\rangle$. In a geometric realization we may write $r=-j|_P$. Let
\[
q:\wtC\longrightarrow C_0:=\wtC/\langle j\rangle.
\]

The following proposition is the consequence of $q$ being a ramified double cover.

\begin{proposition}\label{prop:reflectionJ}
The abelian subvariety
\[
A_r:=\operatorname{Im}(1-r)\subset P
\]
is equal to $q^*JC_0$. The pullback $q^*:JC_0\to J\wtC$ is injective and
\[
\Xi|_{A_r}=2\Theta_{C_0}.
\]
Consequently $(JC_0,\Theta_{C_0})$, and hence $C_0$, is intrinsically recovered from $(P,\Xi)$ together with an involution.
\end{proposition}

After choosing a generator $\sigma$ of the intrinsic subgroup $\mathbb{Z}/d\mathbb{Z}  \subset D_d$, the endomorphisms
\[
\beta_i:=(\sigma^i+\sigma^{-i})|_{A_r},\qquad 1\le i\le k,
\]
are intrinsic as well: $\sigma^i+\sigma^{-i}$ commutes with $j$, hence preserves $A_r$. Note that a different generator only permutes the unordered pairs $\{i,-i\}$, so we have established the following.

\begin{proposition}
    The map $\cP_g^2[d]$ is injective.
\end{proposition}

\section{The theta-dual and the proof of Theorem \ref{thm:mainprime}}\label{sec:theta}

Fix a quotient $C_0$ as above and set
\[
n:=g(C_0)=k(g-1).
\]
The quotient by the full dihedral group $D_d$ induces a degree-$d$ map
\[
h:C_0\longrightarrow \PP^1,
\qquad A:=h^*\cO_{\PP^1}(1).
\]
Fix $x\in C_0$ and a lift $x'\in\wtC$. For $p\in C_0$, choose a lift $p'\in\wtC$ and write $p_a\in C_0$ for the image of $\sigma^a(p')$; define $x_a$ similarly. For $1\le i\le k$ put
\[
D_i:=x_i+x_{d-i}.
\]
Then
\[
\beta_i([p-x])=[p_i+p_{d-i}-D_i].
\]
For $r\ge0$, write
\[
W_r(C_0):=\{L\in\Pic^r(C_0):h^0(C_0,L)>0\},
\]
and set $W_r(C_0)=\varnothing$ for $r<0$. Define the set-theoretic theta-dual
\[
T_i':=\{\xi\in\Pic^{n-1}(C_0):\beta_i(C_0)+\xi\subset W_{n-1}(C_0)\}.
\]

\begin{proposition}\label{prop:thetaexact}
Put $b:=n+1-d=k(g-3).$ Then, set-theoretically,
\[
T_i'=
\bigl(D_i+W_{n-3}(C_0)\bigr)
\cup
\bigl(D_i+K_{C_0}-A-W_b(C_0)\bigr).
\]
The second term is the image of $W_b(C_0)$ under
\[
B\longmapsto D_i+K_{C_0}-A-B.
\]
\end{proposition}

\begin{proof}
Let $\xi\in T_i'$ and put
\[
M:=\xi-D_i\in\Pic^{n-3}(C_0).
\]
By definition, for every $p\in C_0$,
\begin{equation}\label{eq:theta_condition}
h^0\bigl(C_0,M+p_i+p_{d-i}\bigr)>0.
\end{equation}
If $h^0(C_0,M)>0$, then $M\in W_{n-3}(C_0)$ and $\xi\in D_i+W_{n-3}(C_0)$.

Assume $h^0(C_0,M)=0$ and set
\[
L:=K_{C_0}-M.
\]
Since $\deg M=n-3$, the Riemann-Roch formula gives
\[
\deg L=n+1,\qquad h^0(C_0,L)=2.
\]
Applying the Riemann-Roch formula to the degree-$(n-1)$ line bundle in \eqref{eq:theta_condition} gives
\begin{equation}\label{eq:Lcondition}
h^0\bigl(C_0,L-p_i-p_{d-i}\bigr)>0
\qquad\text{for every }p\in C_0.
\end{equation}
Write the complete pencil as
\[
|L|=B+|L_0|,
\]
where $B\ge0$ is the fixed divisor and $|L_0|$ is a complete base-point-free pencil. Choose a general fiber of $h$, away from the branch locus and disjoint from $B$. It consists of $d$ distinct points $p_j$, and $p_i\ne p_{d-i}$ because $d$ is odd and $1\le i\le k$. Condition \eqref{eq:Lcondition} gives a member of the pencil $|L_0|$ containing both points; equivalently,
\[
\varphi_{L_0}(p_i)=\varphi_{L_0}(p_{d-i}).
\]
Replacing $p$ by the points $p_j$ in the same $h$-fiber yields
\[
\varphi_{L_0}(p_{j+i})=\varphi_{L_0}(p_{j-i})
\qquad(j\in\ZZ/d\ZZ).
\]
The graph on $\ZZ/d\ZZ$ whose edges join $j+i$ to $j-i$ is connected because it is generated by translation by $2i$, and $d$ is prime. Hence $\varphi_{L_0}$ is constant on a general fiber of $h$. Consequently
\[
\varphi_{L_0}=\rho\circ h
\]
for a morphism $\rho:\PP^1\to\PP^1$.

Let $r=\deg\rho$. Then
\[
L_0\simeq h^*\cO_{\PP^1}(r)=A^{\otimes r}.
\]
Pullback by the surjective map $h$ gives an injection
\[
H^0(\PP^1,\cO_{\PP^1}(r))\hookrightarrow H^0(C_0,L_0).
\]
The right-hand side has dimension two, while the left-hand side has dimension $r+1$. Since $\varphi_{L_0}$ is nonconstant, $r\ge1$, and therefore $r=1$. Thus
\[
L_0\simeq A,\qquad L\simeq A+B,
\qquad \deg B=n+1-d=b.
\]
It follows that
\[
M=K_{C_0}-A-B,
\qquad
\xi=D_i+K_{C_0}-A-B.
\]

Conversely, let $B\ge0$ have degree $b$ and put
\[
\xi:=D_i+K_{C_0}-A-B.
\]
For every $p\in C_0$, the divisor $A-p_i-p_{d-i}$ is represented by the remaining $d-2$ points, counted with multiplicity, in the $h$-fiber. Hence
\[
h^0\bigl(C_0,A+B-p_i-p_{d-i}\bigr)>0.
\]
The line bundle $K_{C_0}-A-B+p_i+p_{d-i}$ has degree $n-1$, so the Riemann-Roch formula gives
\[
h^0\bigl(C_0,K_{C_0}-A-B+p_i+p_{d-i}\bigr)>0.
\]
Thus, $\xi \in T_i'$. The inclusion $D_i+W_{n-3}(C_0)\subset T_i'$ is immediate.
\end{proof}

\begin{remark}\label{rem:fixedpart}
The factorization argument proves that $d$ divides the degree of the \emph{moving part} of $|L|$, not the degree of $L$ itself. Since
\[
b=n+1-d=k(g-3),
\]
the residual locus $D_i+K_{C_0}-A-W_b(C_0)$ is nonempty for every \(g\geq 3\). If \(g=3\), then
\(b=0\) and \(W_0(C_0)=\{\mathcal O_{C_0}\}\), so the residual locus consists
of the single point \(D_i+K_{C_0}-A\). If \(g>3\), then
\(0<b<n\), and hence \(\dim W_b(C_0)=b=k(g-3)>0\). 

It remains possible a priori that the residual locus is contained in the Brill-Noether translate. The following example shows that this need not occur. Consider $(d,g)=(5,5)$, so that $n=8$ and $b=4$. We first note that we always have $h^0(C_0,A)=2.$ Indeed, $A$ is a base-point-free line bundle of degree five, and Clifford's inequality gives $h^0(A)\leq 3$. If $h^0(A)=3$, then the complete linear series $|A|$ defines a nondegenerate morphism $C_0\to\PP^2$. Since $\deg A=5$ is prime, this morphism is birational onto a plane quintic. This is impossible, since the normalization of a plane quintic has genus at most six, whereas $g(C_0)=8$. Hence $h^0(A)=2$.

By the Riemann-Roch formula, $h^0(K_{C_0}-A)=4$. Therefore, for a general effective divisor $B$ of degree four, $h^0(K_{C_0}-A-B)=0.$ For such $B$, $$\xi=D_i+K_{C_0}-A-B$$ belongs to $T_i'$ by Proposition \ref{prop:thetaexact}, while $\xi\notin D_i+W_5(C_0),$ since $K_{C_0}-A-B$ is not effective. Thus the residual locus $$D_i+K_{C_0}-A-W_4(C_0)$$ is not contained in the Brill-Noether component $D_i+W_5(C_0)$.

\end{remark}

The following lemma is classical.

\begin{lemma}\label{lem:stabilizer}
Let $C_0$ have genus $n$, and let $0\le r\le n-1$. If
\[
W_r(C_0)+a=W_r(C_0)
\]
for some $a\in JC_0$, then $a=0$.
\end{lemma}

\begin{proof}
Adding $W_{n-1-r}(C_0)$ to both sides gives
\[
W_{n-1}(C_0)+a=W_{n-1}(C_0).
\]
The divisor $W_{n-1}(C_0)$ is a theta divisor defining the principal polarization of $JC_0$, whose translation stabilizer is trivial.
\end{proof}

\begin{corollary}\label{cor:maxcomponent}
Assume $d\ge5$ and $n\ge4$. Then $D_i+W_{n-3}(C_0)$ is the unique irreducible component of $T_i'$ of maximal dimension, and $D_i$ is uniquely determined.
\end{corollary}

\begin{proof}
The variety $W_{n-3}(C_0)$ is irreducible of dimension $n-3$. If $b<0$, the residual locus is empty. If $b\ge0$, it is an image of $W_b(C_0)$ and therefore has dimension at most
\[
b=n+1-d.
\]
Since
\[
(n-3)-(n+1-d)=d-4>0,
\]
the first locus is the unique maximal-dimensional component. Uniqueness of the translation follows from Lemma \ref{lem:stabilizer} with $r=n-3$.
\end{proof}

The following observation is essentially \cite[Remark 4.4]{NOPS}. We include the short argument for completeness.

\begin{lemma}\label{lem:effectiveDi}
Assume $g(C_0) \geq 2$. For a general choice of $x\in C_0$, each class $D_i\in\Pic^2(C_0)$ has a unique effective representative.
\end{lemma}

\begin{proof}
Suppose otherwise. Since $D_i$ has degree two, $h^0(C_0,D_i)\ge2$ means that $C_0$ is hyperelliptic and $D_i$ belongs to its unique $g^1_2$. If this occurred for a general $x$ for some fixed $i$, then as $x$ varies in a dense open set the divisors
\[
D_i(x)=x_i+x_{d-i}
\]
would all be linearly equivalent. But
\[
\beta_i([p-x])=[D_i(p)-D_i(x)],
\]
so $\beta_i$ would vanish on the Abel-Jacobi curve and hence on $JC_0$. This is impossible as $\beta_i$ is an automorphism of $JC_0$ for all $i$.
\end{proof}

We now prove Theorem \ref{thm:mainprime}.

\begin{proof}[Proof of Theorem \ref{thm:mainprime}]
Since $d$ is odd, $(d-1)(g-1)$ is even; hence the hypothesis $(d-1)(g-1)\ge7$ actually forces $(d-1)(g-1)\ge8$. Thus
\[
n=\frac{(d-1)(g-1)}2\ge4.
\]
Starting from the polarized Prym $(P,\Xi)$, Proposition \ref{prop:gamma} recovers the intrinsic dihedral group $D_d$ and its unique cyclic subgroup of order $d$. Choose an involution $r$ and a generator $\sigma$. Proposition \ref{prop:reflectionJ} recovers the principally polarized Jacobian $(JC_0,\Theta_{C_0})$ and hence $C_0$, while the maps $\sigma^i+\sigma^{-i}$ recover $\beta_1,\ldots,\beta_k$.

Choose a general point $x\in C_0$ and use it for the Abel-Jacobi embedding. The sets $T_i'$ are then intrinsic. By Corollary \ref{cor:maxcomponent}, their unique maximal-dimensional components recover the divisor classes
\[
D_i=x_i+x_{d-i}\qquad(1\le i\le k).
\]
By Lemma \ref{lem:effectiveDi}, for general $x$ these classes recover the actual effective divisors. Therefore we recover
\[
h^{-1}(h(x))
=x+D_1+\cdots+D_k.
\]
Since $x$ can be chosen arbitrarily in a dense open subset of $C_0$, the construction recovers the general fibers of $h$ and hence the degree-$d$ morphism
\[
h:C_0\to\PP^1
\]
up to an automorphism of $\PP^1$. The Galois closure of $h$ is $\wtC\to\PP^1$ with group $D_d$. The unique cyclic subgroup of order $d$ then recovers
\[
C=\wtC/\langle\sigma\rangle
\]
and the original cyclic cover $\wtC\to C$. Thus the Prym map is injective.
\end{proof}

\subsection*{Exceptional case $d=3$}

Now, let us study the case $d=3$.
We have $k=1$, $n=g-1$, and
\[
h^{-1}(h(x))=x+x_1+x_2,
\qquad
D_1=x_1+x_2\sim A-x.
\]
It turns out that the residual locus of Proposition \ref{prop:thetaexact} is no longer smaller than the expected component.

\begin{proposition}\label{prop:d3theta}
Assume $d=3$ and $n\ge4$. Then
\[
T_1'=K_{C_0}-x-W_{n-2}(C_0)
\]
set-theoretically. 
\end{proposition}

\begin{proof}
Proposition \ref{prop:thetaexact} gives
\[
T_1'=\bigl(D_1+W_{n-3}(C_0)\bigr)
\cup
\bigl(D_1+K_{C_0}-A-W_{n-2}(C_0)\bigr).
\]
Since $D_1\sim A-x$, the second locus equals
\[
K_{C_0}-x-W_{n-2}(C_0).
\]
It remains to show that it contains the first one. Let $E\ge0$ have degree $n-3$. Since $A$ is a trigonal pencil,
\[
h^0(C_0,A+E)\ge2,
\qquad
\deg(A+E)=n.
\]
The Riemann-Roch formula gives
\[
h^0(C_0,K_{C_0}-A-E)\ge1.
\]
Choose $B\ge0$ with
\[
B\sim K_{C_0}-A-E,
\qquad \deg B=n-2.
\]
Then
\[
D_1+E\sim A-x+E\sim K_{C_0}-x-B,
\]
which proves
\[
D_1+W_{n-3}(C_0)\subset K_{C_0}-x-W_{n-2}(C_0).
\]
\end{proof}

\begin{remark}
This explains exactly why the divisibility step in \cite[Proposition 4.3]{NOPS} fails in degree three. A line bundle
\[
L=A+B
\]
with fixed divisor $B$ defines the same map as the moving pencil $|A|$, so its associated morphism factors through $h$ although $3$ need not divide $\deg L$. Only the degree of the moving part is forced to be divisible by $3$.
\end{remark}

\section{Triple covers and trigonal pencils}\label{sec:d3corr}

For $d=3$ we have a much simpler description of the polarized Prym variety.

\begin{proposition}\label{prop:cubicfactor}
Let $d=3$. There is an isomorphism of polarized abelian varieties
\[
(P,\Xi)\simeq
\left(
JC_0\times JC_0,
\begin{pmatrix}
2\lambda_{C_0}&-\lambda_{C_0}\\
-\lambda_{C_0}&2\lambda_{C_0}
\end{pmatrix}
\right).
\]
Consequently the polarized Prym depends only on $(JC_0,\Theta_{C_0})$. Conversely, $(P,\Xi)$ determines $C_0$.
\end{proposition}

\begin{proof}
Recall that $P$ is isomorphic to $JC_0\times JC_0$ with the polarization given by
\[
\begin{pmatrix}
2\lambda_{C_0}&\lambda_{C_0}\beta_1\\
\lambda_{C_0}\beta_1&2\lambda_{C_0}
\end{pmatrix}.
\]
For $d=3$ one has
\[
(1+\sigma+\sigma^2)_{|P}=0_{|P}.
\]
Indeed, on $J\wtC$ one has $f^*\Nm_f=1+\sigma+\sigma^2$, while $\Nm_f$ vanishes on $P=\ker(\Nm_f)^0$. Hence
\[
\beta_1=\sigma+\sigma^{-1}=\sigma+\sigma^2=-1.
\]
This proves the displayed formula and shows that the polarized Prym depends only on $C_0$. The converse follows from Proposition \ref{prop:reflectionJ}: from $(P,\Xi)$ one recovers the Jacobian $(JC_0,\Theta_{C_0})$ and then $C_0$ by the Torelli theorem.
\end{proof}

There is also a direct converse to the dihedral construction. 

\begin{proposition}[Trigonal-dihedral correspondence]\label{prop:trigonalcorrespondence}
Fix $g\ge2$ and put $n=g-1$. Isomorphism classes of connected \'etale cyclic triple covers of hyperelliptic genus-$g$ curves are naturally equivalent to isomorphism classes of simply branched degree-three maps
\[
h:C_0\longrightarrow\PP^1
\]
with $g(C_0)=n$.
\end{proposition}

\begin{proof}
Starting with an \'etale cyclic triple cover $\wtC\to C$, the composite
\[
\wtC\to C\to\PP^1
\]
is Galois with group
$D_3\simeq S_3.$ Then $h:C_0=\wtC/\langle j\rangle\to\PP^1$ is a simply branched cover of degree 3.

Conversely, let $h:C_0\to\PP^1$ be a connected simply branched triple cover. Its monodromy is a transitive subgroup of $S_3$ containing a transposition, hence is all of $S_3$. Let
\[
\wtC\to\PP^1
\]
be its Galois closure and put
\[
C:=\wtC/A_3,
\]
where $A_3\simeq\ZZ/3\ZZ$ is the unique subgroup of order three. Every inertia group is generated by a transposition and therefore intersects $A_3$ trivially. Thus
\[
\wtC\to C
\]
is an \'etale cyclic triple cover.

If $g(C_0)=n$, the Riemann-Hurwitz for the simply branched triple cover gives $2n+4$ branch points. The double cover $C\to\PP^1$ is branched at those same points, hence
\[
g(C)=\frac{(2n+4)-2}{2}=n+1=g.
\]
Thus $C$ is hyperelliptic of genus $g$. The constructions are inverse up to isomorphism.
\end{proof}

Let
\[
\Psi_g:\RH_g[3]\longrightarrow  \cM_{g-1},
\qquad
[f]\longmapsto[C_0]
\]
be the associated map. Proposition \ref{prop:trigonalcorrespondence} shows that the closure of its image is the trigonal locus; this is also the description used by Albano-Pirola \cite{AP}. Proposition \ref{prop:cubicfactor} gives a factorization
\[
\mathcal P_{g}[3]=\Phi_{g-1}\circ\Psi_g
\]
where
\[
\Phi_n(C_0)=
\left(
JC_0\times JC_0,
\begin{pmatrix}
2\lambda_{C_0}&-\lambda_{C_0}\\
-\lambda_{C_0}&2\lambda_{C_0}
\end{pmatrix}
\right).
\]
Moreover, $\Phi_n$ is injective on the image of $\Psi_g$ because the polarized abelian variety recovers $C_0$ by Proposition \ref{prop:cubicfactor}. Thus the fibers of $\cP_g[3]$ are exactly the fibers of $\Psi_g$.

Let $\mathcal H^{\mathrm{sb}}_{3,n}$ denote the Hurwitz space of simply branched degree-three maps $h:C_0\to\PP^1$ with $g(C_0)=n$, modulo isomorphisms of source and target. Let
\[
\nu_n:\mathcal H^{\mathrm{sb}}_{3,n}\longrightarrow \mathcal M_n
\]
be the forgetful map for $n\ge2$. For $n=1$, we interpret the target as
the coarse $j$-line parametrizing smooth genus-one curves. We summarise what we have achieved so far in the following.

\begin{theorem}\label{thm:cubicfactorization}
Put $n=g-1$. The trigonal-dihedral construction gives a natural identification on geometric points
\[
\RH_g[3]\simeq \mathcal H^{\mathrm{sb}}_{3,n}.
\]
Under this identification $\cP_g[3]$ factors as
\[
\mathcal H^{\mathrm{sb}}_{3,n}
\xrightarrow{\ \nu_n\ }
\mathcal M_n
\xrightarrow{\ \Phi_n\ }
\mathcal A^{\delta}_{2n},
\]
where
\[
\Phi_n(C_0)=
\left(
JC_0\times JC_0,
\begin{pmatrix}
2\lambda_{C_0}&-\lambda_{C_0}\\
-\lambda_{C_0}&2\lambda_{C_0}
\end{pmatrix}
\right).
\]
The map $\Phi_n$ is injective on the image of $\nu_n$. Consequently each fiber of $\cP_g[3]$ is naturally identified with a fiber of $\nu_n$, i.e. with the set of isomorphism classes of simply branched trigonal maps having the same source curve $C_0$.
\end{theorem}

\begin{proof}
Proposition \ref{prop:trigonalcorrespondence} gives the identification between cyclic triple covers and simply branched trigonal covers. Proposition \ref{prop:cubicfactor} gives the displayed expression for the polarized Prym and shows that it depends only on $C_0$. Conversely, Proposition \ref{prop:reflectionJ}, together with the intrinsic dihedral group of Proposition \ref{prop:gamma}, recovers the principally polarized Jacobian $(JC_0,\Theta_{C_0})$ from a polarized Prym in the image of the Prym map. The Torelli Theorem then recovers $C_0$. Hence two points in the Hurwitz space have the same Prym if and only if they have isomorphic source curves.
\end{proof}

We now count the trigonal pencils on the general curve occurring in the image of $\Psi_g$. The following dimension count is essentially due to Albano-Pirola: case $n=3$ is
\cite[Section~5]{AP}, where the fiber is identified with an open subset of $C_0$
itself, and case $n=2$ is \cite[Section~6]{AP}; case $n=1$ is immediate from the
dimension count in the proof of \cite[Proposition~2.7]{AP}. We include the short
arguments for completeness.

\begin{lemma}\label{lem:lowgenus}
The general fibers of $\nu_n$ have dimensions
\[
\dim \nu_1^{-1}(C_0)=2,\qquad
\dim \nu_2^{-1}(C_0)=2,\qquad
\dim \nu_3^{-1}(C_0)=1.
\]
Equivalently, the general fibers of $\mathcal P_{g}[3]$ have dimensions $2,2,1$
for $g=2,3,4$, respectively.
\end{lemma}

\begin{proof}
By Theorem~\ref{thm:cubicfactorization}, it is enough to compute the dimension of
the family of simply branched degree-three maps on a fixed general curve $C_0$ of
genus $n$.

Suppose first that $n=1$. Since $\deg L=3>0=2n-2$, every $L\in\Pic^3(C_0)$
satisfies $h^0(C_0,L)=3$. A $g^1_3$ is therefore given by a pair
\[
(L,V),\qquad L\in\Pic^3(C_0),\quad
V\in\operatorname{Gr}(2,H^0(C_0,L))\simeq\PP^2,
\]
so such pairs form a $3$-dimensional family. A subspace $V$ has a base point $p$
precisely when $V=H^0(C_0,L(-p))$, which happens on a one-dimensional locus; hence
the base-point-free pairs form a nonempty open subset, and imposing simple
branching gives a further nonempty open subset. Since maps are considered up to
automorphisms of the source, we quotient by the translation group of $C_0$, hence the general fiber has
dimension $2$.

Let $n=2$. By the Riemann-Roch formula, every $L\in\Pic^3(C_0)$ defines a complete $g^1_3$. The
base-point-free ones form a nonempty open subset of $\Pic^3(C_0)$, and imposing
simple branching gives a further nonempty open subset. Since $\dim\Pic^3(C_0)=2$,
the general fiber has dimension $2$.

Finally, let $n=3$ and let $C_0$ be general, hence nonhyperelliptic. If $L$ is a
$g^1_3$, the Riemann-Roch formula gives
\[
h^0(K_{C_0}-L)=h^0(L)-1\ge1 .
\]
Since $\deg(K_{C_0}-L)=1$, there is a point $p\in C_0$ with $L\simeq K_{C_0}(-p)$,
for a unique $p$. Conversely, $K_{C_0}(-p)$ is
the degree-three pencil obtained by projecting the canonical plane quartic
$C_0\subset\PP^2$ from $p$. Thus the pencils are parametrised by $p\in C_0$, and
the general one is simply branched, the exceptions being the points lying on a
flex tangent of the quartic \cite[Section~5]{AP}. Hence the general fiber has
dimension $1$.
\end{proof}

The following lemma is classical. The first assertion follows from the Castelnuovo-Severi inequality, while the genus-four statement follows from the canonical model of a nonhyperelliptic genus-four curve as the intersection of a quadric and a cubic (see the proof of Proposition \ref{prop:partner} for details).

\begin{lemma}\label{lem:uniqueg13}

A smooth curve $C_0$ of genus at least five admits at most one base-point-free $g^1_3$.

A general smooth curve $C_0$ of genus four has exactly two $g^1_3$'s, and both are simply branched.

\end{lemma}

We are ready to prove Theorem \ref{thm:cubicmain}.

\begin{proof}[Proof of Theorem \ref{thm:cubicmain}]
The cases $g=2,3,4$ are Lemma \ref{lem:lowgenus}.

Let $g=5$. Then $C_0$ has genus four. A general genus-four curve has exactly two simply branched $g^1_3$'s by Lemma~\ref{lem:uniqueg13}. Proposition \ref{prop:trigonalcorrespondence} associates to these two pencils two cyclic \'etale triple covers. For a general genus-four curve $\Aut(C_0)$ is trivial, so the two pencils cannot be identified by an automorphism of $C_0$; hence the two covers are distinct in $\RH_5[3]$. Therefore
\[
\deg \mathcal P_{5}[3]=2.
\]

Finally let $g\ge6$. Then $g(C_0)=g-1\ge5$. The curve $C_0$ carries the base-point-free trigonal pencil by Proposition \ref{prop:trigonalcorrespondence}; in particular, it is nonhyperelliptic. Thus every $g^1_3$ on $C_0$ is base-point-free, and Lemma \ref{lem:uniqueg13} shows that the trigonal pencil is unique. Hence $\Psi_g$ is injective. Since $\Phi_{g-1}$ is injective on the image, $\cP_g[3]$ is injective as well.
\end{proof}

\begin{proposition}\label{prop:partner}
Let $C_0$ be a nonhyperelliptic genus-four curve and let $L$ be a base-point-free $g^1_3$ on $C_0$. Then
\[
L^{\mathrm{res}}:=K_{C_0}\otimes L^{-1}
\]
is again a base-point-free $g^1_3$. On the open locus where the canonical quadric containing $C_0\subset\PP^3$ is smooth, the two pencils $L$ and $L^{\mathrm{res}}$ are distinct and are exactly the two rulings of the quadric. Thus
\[
\tau:(C_0,L)\longmapsto (C_0,K_{C_0}\otimes L^{-1})
\]
is the (rational) involution exchanging the two points in the general fiber of $\mathcal P_{5}[3]$.

The involution has a fixed point precisely when
\[
L^{\otimes2}\simeq K_{C_0}.
\]
On the nonhyperelliptic genus-four locus this is equivalent to the canonical quadric being singular, or equivalently to $C_0$ having a vanishing theta-null.
\end{proposition}

\begin{proof}
Since $\deg L=3$ and $g(C_0)=4$, the Riemann-Roch formula gives
\[
h^0(C_0,K_{C_0}\otimes L^{-1})=h^0(C_0,L)=2.
\]
Since $C_0$ is not hyperelliptic, $L^{\mathrm{res}}$ is a base-point-free $g^1_3$. In the canonical embedding, a genus-four curve is the complete intersection of the unique quadric $Q$ with a cubic. Assume that the canonical quadric $Q$ is smooth. Then
\[
Q\simeq \PP^1\times \PP^1.
\]
Since the canonical embedding of $C_0$ is the complete intersection of
$Q$ with a cubic surface, we have
\[
C_0\in \left|\mathcal O_Q(3,3)\right|.
\]
The two projections
\[
\pi_1,\pi_2:Q\to\PP^1
\]
restrict to degree-three morphisms on $C_0$, hence define two
base-point-free trigonal pencils
\[
L_1:=\mathcal O_{C_0}(1,0),
\qquad
L_2:=\mathcal O_{C_0}(0,1).
\]
Indeed, a fiber of either ruling meets a curve of bidegree $(3,3)$ in
three points.

By adjunction,
\[
K_{C_0}
\simeq
\left(K_Q\otimes\mathcal O_Q(C_0)\right)\big|_{C_0}.
\]
Since
\[
K_Q\simeq\mathcal O_Q(-2,-2)
\qquad\text{and}\qquad
\mathcal O_Q(C_0)\simeq\mathcal O_Q(3,3),
\]
it follows that
\[
K_{C_0}
\simeq
\mathcal O_{C_0}(1,1)
\simeq
L_1\otimes L_2.
\]
Thus
\[
L_2\simeq K_{C_0}\otimes L_1^{-1},
\]
so the involution
\[
L\longmapsto K_{C_0}\otimes L^{-1}
\]
exchanges the two trigonal pencils cut out by the two rulings of $Q$. By Theorem \ref{thm:cubicfactorization}, these two pencils correspond exactly to the two cyclic covers in the general Prym fiber.

A fixed point satisfies $L\simeq K_{C_0}\otimes L^{-1}$, equivalently $L^2\simeq K_{C_0}$. For a nonhyperelliptic genus-four curve this occurs exactly when the unique canonical quadric is a cone (see, for example, \cite[p. 206]{ACGH}); the corresponding $g^1_3$ is then a theta characteristic with two sections, i.e. a vanishing theta-null. 
\end{proof}

\section*{Statement on AI use}
ChatGPT 5.6 Pro and Claude Opus 5 were used to assist with exploratory work, including suggesting ideas, carrying out auxiliary calculations, and assisting with the drafting and revision of the manuscript. The author directed
the overall proof strategies and the development of the mathematical arguments.
The author independently verified all mathematical claims and references, and takes full responsibility for the final manuscript.

\bibliographystyle{plain}
\bibliography{references}

\printaddress

\end{document}